\documentclass[a4paper,11pt,reqno]{amsart}
\usepackage{marginnote,amsmath,amsfonts,amscd,amssymb,graphicx,mathrsfs,eufrak,dsfont}

\usepackage[dvips,all,arc,curve,color,frame]{xy}
\usepackage[usenames]{color}

\usepackage[colorlinks]{hyperref}
\usepackage{tikz,mathrsfs}
\usetikzlibrary{arrows,decorations.pathmorphing,decorations.pathreplacing,positioning,shapes.geometric,shapes.misc,decorations.markings,decorations.fractals,calc,patterns}

\newtheorem{thm}{Theorem}[section]
\newtheorem{prop}[thm]{Proposition}
\newtheorem{lemma}[thm]{Lemma}
\newtheorem{claim}[thm]{Claim}

\newtheorem{obs}[thm]{Observation}

\theoremstyle{definition}

\newtheorem{conj}[thm]{Conjecture}

\begin{document}
\title{Radius, Girth and Minimum degree}

\author{Vojt\u{e}ch Dvo\u{r}\'ak}
\address[Vojt\u{e}ch Dvo\u{r}\'ak]{Department of Pure Maths and Mathematical Statistics, University of Cambridge, UK}
\email[Vojt\u{e}ch Dvo\u{r}\'ak]{vd273@cam.ac.uk}

\author{Peter van Hintum}
\address[Peter van Hintum]{Department of Pure Maths and Mathematical Statistics, University of Cambridge, UK}
\email[Peter van Hintum]{pllv2@cam.ac.uk}

\author{Amy Shaw}

\author{Marius Tiba}
\address[Marius Tiba]{Department of Pure Maths and Mathematical Statistics, University of Cambridge, UK}
\email[Marius Tiba]{mt576@cam.ac.uk}
\begin{abstract}
Given a connected graph $G$ on $n$ vertices, with minimum degree $\delta\geq 2$ and girth at least $g \geq 4$, what is the maximum radius $r$ this graph can have? Erd\H{o}s, Pach, Pollack and Tuza \cite{Erdos} established in the triangle-free case ($g=4$) that $r \leq \frac{n-2}{\delta}+12$, and noted that up to the value of the additive constant, this is tight. We determine the exact value for the triangle-free case.\\
For higher $g$ little is known. We settle the order of $r$ for $g=6,8,12$ and prove an upper bound to the order for general even $g$. Finally, we show that proving the corresponding lower bound for general even $g$ is equivalent to the Erd\H{o}s girth conjecture.
\end{abstract}

\maketitle
\parindent 20pt
\parskip 0pt

\section{Introduction}\label{intro}

Consider the following question: given a connected graph $G$ on $n$ vertices, with minimum degree $\delta\geq 2$ and girth at least $g \geq 4$, what is the maximum radius $r$ this graph can have?
Denote this maximum value of the radius\footnote{for those triples $n, \delta,g$ with $\delta \geq 2$ for which there exists a connected graph on $n$ vertices, with minimum degree $\delta$ and of girth at least $g$} as $r(n, \delta,g)$.

Erd\H{o}s, Pach, Pollack and Tuza \cite{Erdos} studied this problem for the case $g=4$. They established that $r(n,\delta,4) \leq \frac{n-2}{\delta}+12$, and noted that up to the value of the additive constant, this is tight.

Our first result settles the triangle-free case fully.

\begin{thm}\label{main}
Consider any integers $\delta \geq 2$ and $n$. If $n<2\delta$, there exists no connected triangle-free graph on $n$ vertices with minimum degree $\delta$. If $n \in \left\{ 2\delta,2\delta+1 \right\}$, we have $r(n,\delta,4)=2$. If $2\delta+2 \leq n <4\delta$, we have $r(n,\delta,4)=3$. Most importantly, if $n \geq 4 \delta$, then $$r(n, \delta, 4)=\begin{cases}\frac{n}{\delta}-1 &\text{if $\delta$ is odd and $n=k \delta$ for $k$ odd}\\
\lfloor \frac{n}{\delta} \rfloor &\text{otherwise}
\end{cases}$$
\end{thm}

Next we consider the case when the girth $g$ is bigger than $4$. We obtain a general upper bound.

\begin{thm}\label{upbound}
Let $n,\delta \geq 2$, and $g=2k$  such that there exists a connected graph on $n$ vertices with minimum degree $\delta$ and girth at least $g$. Then $$r(n,\delta,g) \leq \frac{n k}{2 \delta(\delta-1)^{k-2}}+ 3 k.$$
\end{thm}

We also show that in the cases $g=6,8,12$, this is essentially best possible.

\begin{thm}\label{specialgirths}
Let $\delta \geq 2$ so that $\delta-1$ is a prime power.

Then $\exists n_i\to \infty$ so that $r(n_i,\delta,6)\geq \frac{3n_i}{2(\delta^2-\delta+1)}-3$.
Further, $\exists n_j\to \infty$ so that $r(n_j,\delta,8)\geq \frac{2n_j}{\delta^{3}-2\delta^{2}+2\delta}-4$.
Finally, $\exists n_k\to \infty$ so that $r(n_k,\delta,12)\geq \frac{3n_k}{((\delta-1)^{3}+1)(\delta^{2}-\delta+1)}-6$.
\end{thm}

Note that in the result above, not only do the coefficients of the leading term in $\delta$ agree with the upper bound we have from Theorem \ref{upbound}, but even the coefficients of the second order terms agree.

It would be interesting to see whether the upper bound from Theorem \ref{upbound} is tight, at least up to some constant factor. As our final result, we obtain the following proposition.

\begin{prop}\label{erdos}
Let $r,c>0$, $g=2k$ and $n \leq c(r+1) \delta^{k-1}$, so that $r(n, \delta,g) \geq r$. Then there exists a connected graph of girth at least $2k$ on at most $(2k+1)c \delta^{k-1}$ vertices with at least $\frac{1}{2}\delta^{2}(\delta-1)^{k-2}$ edges.
\end{prop}

This relates the question whether the upper bound from Theorem \ref{upbound} is tight up to some constant factor to the following girth conjecture of Erd\H{o}s \cite{Erdosgirth}.

\begin{conj} (Erd\H{o}s)
For any positive integers $l,n$, there exists a graph with girth $2l+1$, $n$ vertices and $\Omega(n^{1+\frac{1}{l}})$ edges.
\end{conj}

We see that if for some fixed $g=2k$ and fixed $c>0$, we could find graphs $G_i$ with $\delta_i \to \infty$ and $n_i \leq c(r(n_i,\delta_i,2k)+1) \delta_i^{k-1}$, then by Proposition \ref{erdos}, that would verify the girth conjecture of Erd\H{o}s for $l=k-1$.

The structure of the paper is as follows: in Section \ref{setup}, we establish a key technical lemma. In the next two sections, we consider separately the triangle-free case and the general case.

\section{Strategy}\label{setup}

Throughout the paper, we will use the following lemma as a useful tool. It tells us that if we can find a large collection of vertices in our graph such that any two elements are either neighbours or sufficiently far away from each other, then our graph must in fact have many vertices.

\begin{lemma}\label{generaltool}
Assume $G$ is a graph on $n$ vertices of girth $g \geq 2k$ (where $k \geq 2$) with minimum degree $\delta$. Then for any subset $T \subset V(G)$ such that all pairs of non-adjacent vertices in T have distance at least $2k-1$ from each other, we have $n \geq |T|\delta(\delta-1)^{k-2} $. Moreover if $|T|$ is odd, we have $n \geq |T|\delta(\delta-1)^{k-2} +1$.
\end{lemma}

\begin{proof}
For every $v \in T$, let $S(v)= \left\{ w \in V(G) \, \, \, | \, \, \, d(v,w)=k-1 \right\}$. 

First, we claim that for any distinct $v_1,v_2 \in T$, the sets $S(v_1),S(v_2)$ are disjoint. To see that, consider two cases: if $d(v_1,v_2) \geq 2k-1$ and for some $a \in V(G)$ we have $a \in S(v_1) \cap S(v_2)$, then we get $2k-1 \leq  d(v_1,v_2) \leq d(v_1,a)+d(a,v_2)=2k-2$, yielding a contradiction. If $d(v_1,v_2) =1$ and for some $a \in V(G)$ we have $a \in S(v_1) \cap S(v_2)$, let $v_1,c_1,...,c_{k-2},a$ be a path of length $k-1$ from $v_1$ to $a$, and let $v_2,d_1,...,d_{k-2},a$ be a path of length $k-1$ from $v_2$ to $a$. First note we can not have $c_1=v_2$, as that would imply $d(v_2,a) \leq k-2$. Analogously we can not have $d_1=v_1$. But that now implies that the subgraph of $G$ spanned by $v_1,v_2,a,c_1,...,c_{k-2},d_1,...,d_{k-2}$ contains a cycle of length at most $2k-1$, contradicting the assumption that $g \geq 2k$.

Next we claim that for any $v \in T$, we have $|S(v)| \geq \delta(\delta-1)^{k-2}$. Assume for contradiction not. Let $i$ be smallest integer $2 \leq i \leq k$ such that for $S_i(v)=\left\{ w \in V(G) \, \, \, | \, \, \, d(v,w)=i-1 \right\}$, we have $|S_i(v)| < \delta(\delta-1)^{i-2}$. Note that trivially $i \geq 3$, as for $i=2$ this is just the minimum degree condition.

Now, consider any two different $a,b \in S_{i-1}(v)$. Note that every neighbour of $a$ (and analogously of $b$) must be either in $S_{i-2}(v)$, $S_{i-1}(v)$ or $S_{i}(v)$ by definition.

Firstly, note that $a$ and $b$ can not be connected together, as that would contradict the girth assumption. Next, note that $a$ (and analogously $b$) can be connected to at most one vertex in $S_{i-2}(v)$, else we would again contradict girth assumption. Finally, again by girth assumption, note that $a$ and $b$ can not share any neighbour in $S_i(v)$. Altogether, this implies $|S_i(v)| \geq (\delta-1)|S_{i-1}(v)| \geq \delta(\delta-1)^{i-2} $, a desired contradiction.

It now follows $n=|V(G)| \geq |\cup_{v \in T} S(v)|=\sum_{v \in T}|S(v)| \geq |T|\delta(\delta-1)^{k-2}$.

Finally consider the case when $|T|$ is odd. Note that every vertex in $T$ can be at distance $1$ from at most one other vertex of $T$, else we would get either $C_3$ in $T$ or a pair of vertices of $T$ at mutual distance $2$. Hence, if $|T|$ is odd, some vertex in $T$ is at distance at least $2k-1$ from every other vertex in $T$. Hence it was not counted in the calculation above, giving $n \geq |T|\delta(\delta-1)^{k-2} +1$.
\end{proof}

To find such a large collections of points with restricted mutual distances, we will use several observations. Those can be found in Appendix \ref{observ}.

\section{Triangle-Free Graphs}

To prove Theorem \ref{main}, we will establish the following three propositions.

\begin{prop}\label{smallcases}
Fix $\delta \geq 2$ and $n$. Then every connected triangle-free graph on $n$ vertices with radius $r$ satisfies $r \geq 2$ and $n \geq 2 \delta$. Moreover, if $r=3$, we have $n \geq 2 \delta+2$.

Further, for any $n \geq 2\delta$, we have a connected triangle-free graph on $n$ vertices with radius $2$. And for any $n \geq 2\delta+2$, we have a connected triangle-free graph on $n$ vertices with radius $3$.
\end{prop}

\begin{prop}\label{firstbit}
Let $r \geq 4$, $\delta \geq 2$, $c \geq 0$ be integers. Then there exists a connected triangle-free graph with $n=2 \lceil \frac{r \delta}{2} \rceil +c$ vertices, minimum degree $\delta$ and radius $r$.
\end{prop}

\begin{prop}\label{secondbit}
If $G$ is a connected triangle-free graph on $n$ vertices with minimum degree $\delta \geq 2$ and radius $r \geq 4$, then we have $n \geq 2 \lceil \frac{r \delta}{2} \rceil$.
\end{prop}

Let us first see how Theorem \ref{main} follows from these.

\begin{proof}[Proof of Theorem \ref{main} assuming Proposition \ref{smallcases}, Proposition \ref{firstbit} and Proposition \ref{secondbit}]
If $n<4\delta$, it follows from Proposition \ref{secondbit} that $r(n,\delta,4) \leq 3$. The result then follows from Proposition \ref{smallcases}. Assume further $n \geq 4 \delta$.

First consider the case when $\delta$ is odd and $n=k \delta$ for $k$ odd. Using Proposition \ref{firstbit} with $r=\frac{n}{\delta}-1$ and $c=\delta$, we conclude there exists a connected triangle-free graph with $n$ vertices, minimum degree $\delta$ and radius $\frac{n}{\delta}-1$, and hence that $r(n, \delta,4) \geq \frac{n}{\delta}-1$. Also, for any connected triangle-free graph with $n$ vertices and minimum degree $\delta$, it follows in this case from Proposition \ref{secondbit} that its radius $r$ satisfies $r < \frac{n}{\delta}$. Since $r$ is integer, that implies $r \leq \frac{n}{\delta}-1$. Hence, we conclude $r(n, \delta,4) = \frac{n}{\delta}-1$.

Next consider the other case. Using Proposition \ref{firstbit} with $r= \lfloor \frac{n}{\delta} \rfloor$ and $c=n-2 \lceil \frac{r \delta}{2} \rceil$, we conclude there exists a connected triangle-free graph with $n$ vertices, minimum degree $\delta$ and radius $\lfloor \frac{n}{\delta} \rfloor$, and hence that $r(n, \delta,4) \geq \lfloor \frac{n}{\delta} \rfloor$.  Also, for any connected triangle-free graph with $n$ vertices and minimum degree $\delta$, it follows from Proposition \ref{secondbit} that its radius $r$ satisfies $r \leq \frac{n}{\delta}$. Since $r$ is integer, that implies $r \leq \lfloor \frac{n}{\delta} \rfloor$, and hence we conclude $r(n, \delta,4) = \lfloor \frac{n}{\delta} \rfloor$. 
\end{proof}

In the rest of the section, we will prove Propositions \ref{smallcases}, \ref{firstbit} and \ref{secondbit} and thus prove Theorem \ref{main}. The section will be divided into four subsections - in the first subsection we prove Proposition \ref{firstbit}; in the second subsection we prove a technical lemma we will need to prove Proposition \ref{secondbit}; in the third subsection we prove Proposition \ref{smallcases}; in the fourth subsection we prove Proposition \ref{secondbit} when $r=4k$, $r=4k+1$ or $r=4k+2$; and in the final subsection we prove Proposition \ref{secondbit} when $r=4k+3$.

\subsection{Proof of Proposition \ref{firstbit}}

To prove Proposition \ref{firstbit}, it is enough to consider the following simple example.

\begin{proof}[Proof of Proposition \ref{firstbit}]
Consider $2r$ boxes labelled $B_0,...,B_{2r-1}$. Let $|B_i|= \left \lceil{\frac{\delta}{2}}\right \rceil $ when $i \equiv 0,1$ $\mod 4$ and $|B_i|=\left \lfloor{\frac{\delta}{2}}\right \rfloor$ when $i \equiv 2,3 $ $\mod 4$. If $c$ is such that at this point we have less than $2 \lceil \frac{r \delta}{2} \rceil + c$ vertices, put the remaining vertices into any of these boxes arbitrarily. Then connect all vertices in $B_i$ to all vertices in $B_j$ whenever $i-j \equiv \pm 1$ $\mod 2r$. It is now easy to see that this graph has required properties for any choice of parameters $r,\delta,c$ in our range.
\end{proof}

\subsection{Technical lemma}

First, recall Lemma \ref{generaltool} which implies the following result for triangle-free graphs.

\begin{lemma}\label{maintool}
Let $G$ be a triangle-free graph on $n$ vertices and with minimum degree $\delta$. Then for any subset $T \subset V(G)$ such that no two vertices of $T$ are at mutual distance $2$, we have $n \geq 2\left\lceil \frac{\delta |T|}{2} \right\rceil$.
\end{lemma}

We will also need another lemma of similar flavour here.

\begin{lemma}\label{cycletool} 
Let $G$ be a triangle-free graph on $n$ vertices and with minimum degree $\delta$. Assume for some $r \geq 4$, we have a subset $U \subset V(G)$ such that $|U|=2r$ and $U$ is as follows: if we consider auxiliary graph $H$ such that $V(H)=U$ and in which we connect two vertices if their distance in $G$ is precisely $2$, then $H$ is disjoint union of two cycles of length $r$. Then we have $n \geq 2 \left\lceil \frac{r \delta}{2}  \right\rceil $. 
\end{lemma}

\begin{proof}
Let $c_1,\dots,c_r$ and $d_1,\dots,d_r$ be our two cycles of length $r$ in $H$. Consider the open neighbourhoods $N(c_1)$,\dots,$N(c_{r})$. On one hand, we have $|N(c_i)|\geq \delta$ for $1 \leq i \leq r$. On the other hand, each $v \in V(G)$ can be contained in the neighbourhood of at most two vertices from $\left\{ c_1,\dots,c_r \right\}$. Indeed, if $v$ is contained in $N(c_i),N(c_j),N(c_k)$, then as $G$ is triangle-free, we get $d(c_i,c_j)=d(c_i,c_k)=d(c_j,c_k)=2$, so $H$ contains $K_3$, a contradiction. 

Further, no vertex can be contained both in some set of the form $N(c_i)$ and in some set of the form $N(d_j)$. This is true because no $c_i$ and $d_j$ are at mutual distance $2$, and hence by the triangle-free condition can not share a neighbour.

Let $B$ be the set of vertices of $G$ contained in at least one set of the form $N(c_i)$ (and hence no set of the form $N(d_j)$). By above, we find

\begin{equation*}
2|B| \geq |N(c_1)|+...+|N(c_{r})| \geq r \delta.
\end{equation*}

Since $|B|$ is integer, we have $|B| \geq \lceil \frac{r \delta}{2} \rceil$.

Let $B'$ be the set of vertices of $G$ contained in at least one set of the form $N(d_j)$. By analogous argument, we get $|B'| \geq \lceil \frac{r \delta}{2} \rceil$. Using that $B,B'$ are disjoint, we obtain $n \geq 2 \lceil \frac{r \delta}{2} \rceil$. 
\end{proof}

\subsection{Proof of Proposition \ref{smallcases}} Here we handle the small radius cases.

\begin{proof}[Proof of Proposition \ref{smallcases}]
Consider a connected triangle-free graph $G$ on $n$ vertices of radius $r$ and minimum degree $\delta \geq 2$. We must have $r\geq 2$, since the only connected triangle-free graphs of radius $1$ are star graphs, but those have minimum degree $1$.

Now consider any two adjacent vertices $a,b \in V(G)$. It follows from Lemma \ref{maintool} applied to $T= \left\{a,b \right\}$ that $n \geq 2 \delta$. 

If $r=3$, then we can take $a,b$ which instead satisfy $d(a,b) \geq 3$. But then even their closed neighbourhoods are disjoint, which implies $|V(G)| \geq |N[a] \cup N[b]|=|N[a]|+|N[b]| \geq 2\delta+2$.

Now if $n\geq 2 \delta$, we can easily check that $K_{\delta,n-\delta}$ is a connected triangle-free graph on $n$ vertices of radius $2$ and minimum degree $\delta$.

If $n \geq 2\delta+2$, start with a complete bipartite graph $K_{\delta+1,n-\delta-1}$ with vertex classes $\left\{ v_1,...,v_{\delta+1} \right\}$ and $\left\{ w_1,...,w_{n-\delta-1} \right\}$. Erase the edges $v_1 w_1$, $v_2 w_2$,..., $v_{\delta+1}w_{\delta+1}$, $v_{\delta+1}w_{\delta+2}$,...,$v_{\delta+1}w_{n-\delta-1}$. The resulting graph is a connected triangle-free graph on $n$ vertices of radius $3$ and minimum degree $\delta$.
\end{proof}

\subsection{Proof of Proposition \ref{secondbit} for $r=4k,4k+1,4k+2$}

In this subsection, we prove the following.

\begin{prop}\label{mostvals}
If $G$ is a connected triangle-free graph on $n$ vertices with minimum degree $\delta \geq 2$ and radius $r \geq 4$ such that $r=4k+i$ for some $k$ and some $i \in \left\{ 0,1,2 \right\}$, then we have $n \geq  2 \lceil \frac{r \delta}{2}  \rceil$. 
\end{prop}

\begin{proof}
Let $v_0$ be a center of our graph $G$, which moreover has the property that every other center of $G$ is at distance $r$ from at least as many vertices as $v_0$ is. Let $v_r$ be any vertex such that $d(v_0,v_r)=r$. Let $v_0,v_1,...,v_r$ be a path of length $r$ from $v_0$ to $v_r$.

Let $v_{r-t}'$ be a following vertex: if $v_3$ is not a center of $G$, then $v_{r-t}'$ is any vertex such that $d(v_3,v_{r-t}') \geq r+1$. And if $v_3$ is a center of $G$, then we let $v_{r-t}'$ be such a vertex that $d(v_3,v_{r-t}') = r$ and $d(v_0,v_{r-t}') <r$ (such a vertex exists by a choice of $v_0$). 

Then denote $d(v_0,v_{r-t}')=r-t$ for some $t \geq 0$. Let $v_0=v_0',v_1',...,v_{r-t}'$ be a path of length $r-t$ from $v_0$ to $v_{r-t}'$. It follows from Observation \ref{bound} that $t \leq 3$.

\begin{claim}\label{easycases}
Assume that either $r=4k$ (and $t$ is any), or $r=4k+2$ (and $t$ is any), or $r=4k+1$ and $0 \leq t \leq 2$. Then we have $n \geq  2 \lceil \frac{r \delta}{2}  \rceil$.
\end{claim}

\begin{proof}[Proof of Claim \ref{easycases}]
We will show that in each of these cases, we can find a collection $C$ of $r$ vertices in $G$ such that no two are at mutual distance $2$. The result then follows from Lemma \ref{maintool}.

Depending on the values of $r$ and $t$, choose $C$ to be the following collection.

\begin{center}
\begin{tabular}{ |c|c|c| } 
 \hline
 $r=4k$ & $t=0$ & $v_3,v_4,v_7,v_8,...,v_{4k-1}, v_{4k},v_3',v_4',v_7',v_8',...,v_{4k-1}', v_{4k}'$ \\ 
 $r=4k$ & $t=1$ & $v_0,v_3,v_4,v_7,v_8,...,v_{4k-1}, v_{4k},v_3',v_4',v_7',v_8',...,v_{4k-5}', v_{4k-4}',v_{4k-1}'$ \\ 
 $r=4k$ & $t=2$ & $v_3,v_4,v_7,v_8,...,v_{4k-1}, v_{4k},v_1',v_2',v_5',v_6',...,v_{4k-3}', v_{4k-2}'$ \\ 
 $r=4k$ & $t=3$ & $v_0,v_3,v_4,v_7,v_8,...,v_{4k-1}, v_{4k},v_1',v_4',v_5',v_8',v_9',...,v_{4k-4}', v_{4k-3}'$ \\ 
 $r=4k+1$ & $t=0$ & $v_0,v_4,v_5,v_8,v_9,...,v_{4k}, v_{4k+1},v_4',v_5',v_8',v_{9}',...,v_{4k}', v_{4k+1}'$ \\ 
 $r=4k+1$ & $t=1$ &  $v_0,v_3,v_4,v_7,v_8,...,v_{4k-1}, v_{4k},v_3',v_4',v_7',v_{8}',...,v_{4k-1}', v_{4k}'$\\ 
 $r=4k+1$ & $t=2$ & $v_0,v_1,v_4,v_5,v_8,v_{9},...,v_{4k}, v_{4k+1},v_3',v_4',v_7',v_{8}',...,v_{4k-5}',v_{4k-4}',v_{4k-1}' $ \\
 $r=4k+2$ & $t=0$ & $v_0,v_1,v_5,v_6,v_9,v_{10},...,v_{4k+1}, v_{4k+2},v_5',v_6',v_9',v_{10}',...,v_{4k+1}', v_{4k+2}'$ \\ 
 $r=4k+2$ & $t=1$ &  $v_0,v_1,v_4,v_5,v_8,v_{9},...,v_{4k}, v_{4k+1},v_4',v_5',v_8',v_{9}',...,v_{4k}', v_{4k+1}'$\\ 
 $r=4k+2$ & $t=2$ & $v_0,v_1,v_5,v_6,v_9,v_{10},...,v_{4k+1}, v_{4k+2},v_3',v_4',v_7',v_{8}',...,v_{4k-1}', v_{4k}'$ \\ 
 $r=4k+2$ & $t=3$ & $v_1,v_2,v_5,v_6,v_9,v_{10},...,v_{4k+1}, v_{4k+2},v_2',v_3',v_6',v_{7}',...,v_{4k-2}', v_{4k-1}'$ \\ 
 \hline
\end{tabular}
\end{center}

We need to check two things; that $C$ genuinely consists of $r$ distinct vertices, and that no two vertices of $C$ have mutual distance $2$. 

None of the collections above contains both $v_1$ and $v_1'$. For all other pairs $v_i,v_j'$, it follows in our particular case from Observation \ref{distinct} that $v_i \neq v_j'$. So $C$ consists of $r$ distinct vertices.

Note that $v_0,...,v_r$ is a path of length $r$ and $v_0',...,v_{r-t}'$ is a path of length $r-t$. Hence we can trivially check that $C$ contains no two vertices of the form $v_i,v_j$ such that $d(v_i,v_j)=2$ and no two vertices of the form $v_i',v_j'$ such that $d(v_i',v_j')=2$.

Finally, to check that $C$ contains no two vertices of the form $v_i,v_j'$ such that $d(v_i,v_j')=2$, we first note that by Observation \ref{distinct} we would have to have $|i-j| \leq 2$.

First consider the case $i \geq 3$. Note that by our choice of $v_0$, $v_{r-t}'$, we always have either $t \geq 1$ or $d(v_3,v_{r-t}') \geq r+1$. If $j \geq i-1$, it follows from Observation \ref{triangle} that $d(v_i,v_j') \geq 3$. If $j=i-2$, $d(v_i,v_j') \geq 3$ follows from Observation \ref{triangle} under additional assumption that $d(v_3,v_{r-t}')+t \geq r+2$. Hence, for $i \geq 3$, it is enough if $C$ does not contain both $v_i$ and $v_{i-2}'$ in the case when we have $d(v_3,v_{r-t}')+t \leq r+1$.

Next, consider the case $i=2$. It follows from Observation \ref{triangle} that it suffices to ensure that if our collection contains $v_2$, then it does not contain:

\begin{itemize}
    \item $v_1'$ in the case $d(v_3,v_{r-t}')+t \leq r+2$
    \item $v_2'$ in the case $d(v_3,v_{r-t}')+t \leq r+1$
\end{itemize}

Finally, consider the case $i=1$. It follows from Observation \ref{triangle} that it suffices to ensure that if our collection contains $v_1$, then it does not contain:

\begin{itemize}
    \item $v_1'$ in the case $v_1 \neq v_1'$
    \item $v_2'$ in the case $d(v_3,v_{r-t}')+t \leq r+2$
    \item $v_3'$ in the case $d(v_3,v_{r-t}')+t \leq r+1$
\end{itemize}

Recall that if $t=0$, then $d(v_3,v_{r-t}') \geq r+1$. Hence, it can be checked trivially going through the discussion above that in each of the cases, no two vertices in $C$ are at mutual distance $2$. The result follows.
\end{proof}

\begin{claim}\label{hardcase}
If $r=4k+1$ and $t=3$, we have $n \geq  2 \lceil \frac{r \delta}{2}  \rceil$.
\end{claim}

\begin{proof}[Proof of Claim \ref{hardcase}]
We let $v_{r-s}''$ be such a vertex that $d(v_1,v_{r-s}'') \geq r$, then $d(v_0,v_{r-s}'')=r-s$ for some $0 \leq s \leq 1$.

First consider the case when $d(v_{r-s}'',v_{4k-2}') \geq 3$. Then let 
\begin{equation*}
T= \left\{ v_2,v_3,v_6,v_7,...,v_{4k-2},v_{4k-1},v_1',v_2',v_5',v_6',...,v_{4k-3}',v_{4k-2}',v_{r-s}'' \right\}    
\end{equation*}

Assume for a contradiction two vertices of $T$ have mutual distance 2. It follows from Observation \ref{triangle} that one of them has to be $v_{r-s}''$. Since for any $v,w\in V(G)$, we have $d(v,w) \geq |d(v,v_0)-d(w,v_0)|$ and $d(v_{r-s}'',v_{4k-2}') \geq 3$ by assumption, it further follows that the other vertex would have to be $v_{4k-2}$ or $v_{4k-1}$. Note that if $d(v_i,v_{r-s}'') \leq 2$ for some $1 \leq i \leq 4k-1$, then $d(v_1,v_{r-s}'') \leq d(v_1,v_i)+d(v_i,v_{r-s}'') \leq (4k-2)+2<r$, yielding a desired contradiction. Hence, no two vertices of $T$ have mutual distance $2$ while $|T|=r$. The result then follows from Lemma \ref{maintool}.

Next, consider the case $d(v_{r-s}'',v_{4k-2}') < 3$. Since $$d(v_{r-s}'',v_{4k-2}') \geq|d(v_{r-s}'',v_0)-d(v_{4k-2}',v_0)|\geq 3-s \geq 2,$$ this means $s=1$ and $d(v_{r-1}'',v_{4k-2}') =2$. Hence there exists a vertex $a$, such that $a$ is neighbour of both $v_{r-1}''$ and $v_{4k-2}'$. Moreover, clearly $d(a,v_0)=r-2$.

Consider two cases. If $d(a,v_{4k+1}) \geq 3$, take
\begin{equation*}
T= \left\{ v_1,v_2,v_5,v_6,...,v_{4k-3},v_{4k-2},v_{4k+1},v_2',v_3',v_6',v_7',...v_{4k-6}',v_{4k-5}',v_{4k-2}',a \right\}    
\end{equation*}

Assume for a contradiction two vertices of $T$ have mutual distance 2. It follows from Observation \ref{triangle} one of them has to be $a$. Since for any $v,w$ in $G$, we have $d(v,w) \geq |d(v,v_0)-d(w,v_0)|$ and $d(a,v_{4k+1}) \geq 3$ and $d(a,v_{4k-2}')=1$, the other vertex has to be $v_{4k-3}$ or $v_{4k-2}$. Note that if $d(a,v_{i}) \leq 2$ for some $3 \leq i \leq 4k-2$, then $$d(v_3,v_{4k-2}') \leq d(v_3,v_i)+d(v_i,a)+d(a,v_{4k-2}) \leq (4k-5)+2+1<r,$$ a contradiction. Hence, no two vertices of $T$ have mutual distance $2$ and $|T|=r$. The result follows from Lemma \ref{maintool}.

Next, consider the case $d(a,v_{4k+1}) < 3$. By the triangle inequality, we have $d(a,v_{4k+1}) \geq |d(a,v_0)-d(v_0,v_{4k+1})|= 2$, so that $d(a,v_{4k+1}) =2$. Hence, there exists a vertex $b$ such that $b$ is neighbour of both $a$ and $v_{4k+1}$. Consider 
\begin{equation*}
U= \left\{ v_0,v_1,v_2,v_3,...,v_{4k+1},v_{1}',v_2',...,v_{4k-2}',a,b \right\}.   
\end{equation*}

We have $|U|=8k+2=2r$. Consider auxiliary graph $H$ on $V(H)=U$ in which we connect two vertices if their distance in $G$ is precisely $2$. $H$ is union of two disjoint cycles of length $r$, first being $v_0,v_2,...,v_{4k},b,v_{4k-2}',...,v_2'$, and second being $v_1,v_3,...,v_{4k+1},a,v_{4k-3}',...,v_1'$. The result then follows from Lemma \ref{cycletool}. The only non trivial relationships needed to prove that $H$ is union of two disjoint cycles of length $r$ are

\begin{equation*}
d(b,v_{4k-1}),d(b,v_{4k-2}),d(a,v_{4k}),d(a,v_{4k-1}),d(a,v_{4k-2}),d(a,v_{4k-3}) \geq 3 
\end{equation*}

If any of these distances was at most $2$, we could find a path of length at most $r-1$ from $v_3$ to $v_{4k-2}'$. That would be a contradiction. 
\end{proof}

Putting Claim \ref{easycases} and Claim \ref{hardcase} together now finishes the proof of Proposition \ref{mostvals}.
\end{proof}

\subsection{Proof of Proposition \ref{secondbit} for $r=4k+3$}

In this subsection, we prove the following.

\begin{prop}\label{oneval}
If $G$ is a connected triangle-free graph on $n$ vertices with minimum degree $\delta \geq 2$ and radius $r \geq 4$ such that $r=4k+3$ for some $k$, then we have $n \geq  2 \lceil \frac{r \delta}{2}  \rceil$. 
\end{prop}

We use a slightly weaker and more general set-up than we did in the proof of Proposition \ref{mostvals}. This will have the advantage that we have more freedom in our choice of a center $v_0$ as well as in the choice of $v_{r-t}'$.

\begin{proof}
Take $v_0$ to be any center of our graph $G$. Let $v_r$ be any vertex such that $d(v_0,v_r)=r$. Let $v_0,v_1,...,v_r$ be any path of length $r$ from $v_0$ to $v_r$.

Let $v_{r-t}'$ be any vertex such that $d(v_3,v_{r-t}') \geq r$. 

Then we have $d(v_0,v_{r-t}')=r-t$ for some $t \geq 0$. Let $v_0=v_0',v_1',...,v_{r-t}'$ be a path of length $r-t$ from $v_0$ to $v_{r-t}'$. By Observation \ref{bound}, we have $t \leq 3$.

Moreover, consider a vertex $v_{r-s}''$ such that $d(v_4,v_{r-s}'') \geq r$. 

As before, we have $d(v_0,v_{r-s}'')=r-s$ for some $s \geq 0$. Let $v_0=v_0'',v_1'',...,v_{r-s}''$ be a path of length $r-s$ form $v_0$ to $v_{r-s}''$. By Observation \ref{bound}, we have $s \leq 4$.

We will consider four cases depending on the value of $t$.

Consider the case $t=s=0$. Let 
\begin{equation*}
T= \left\{ v_0,v_3,v_3',v_6,v_6'',v_7,v_7'',v_{10},v_{10}'',v_{11},v_{11}'',...,v_{r-5},v_{r-5}'',v_{r-4},v_{r-4}'',v_{r-1},v_{r-1}'',v_r,v_r'' \right\}. 
\end{equation*}
By Observation \ref{triangle}, no two elements of $T$ have mutual distance $2$. The result follows from Lemma \ref{maintool}.\\

Next consider the case $t=0$, $1 \leq s \leq 4$. We claim that we can find four vertices $z_1,z_2,z_3,z_4$ such that no two out of $z_1,z_2,z_3,z_4$ have mutual distance $2$, and $r-3 \geq d(v_0,z_i) \geq r-4$ for $i=1,2,3,4$.

Set $z_1=v_{r-4}$, $z_2=v_{r-3}$, $z_3=v_{r-4}''$. By Observation \ref{triangle}, we immediately see $d(v_{r-4},v_{r-4}'') \geq 5$, $d(v_{r-3},v_{r-4}'') \geq 4$. If we have any vertex $x$ such that $x$ is neighbour of $v_{r-4}''$ and $d(v_0,x) \geq r-4$, we can set $z_4=x$ and are done. If on the other hand there exists no such $x$, that implies $d(v_{r-3}',v_{r-4}'') \geq 3$. By Observation \ref{triangle} we have $d(v_{r-4},v_{r-3}') \geq 4$, $d(v_{r-3},v_{r-3}') \geq 3$, so we can set $z_4=v_{r-3}'$. Hence, we can always find suitable  $z_1,z_2,z_3,z_4$. 
 
Let
\begin{equation*}
T= \left\{ v_0,v_3,v_3'',v_4,v_4'',v_7,v_7'',v_{8},v_{8}'',...,v_{r-8},v_{r-8}'',v_{r-7},v_{r-7}'',z_1,z_2,z_3,z_4,v_r,v_r' \right\}. 
\end{equation*}
It follows from Observation \ref{triangle} that no two elements of $T$ have mutual distance $2$. The result follows from Lemma \ref{maintool}.\\

Consider the case $t=2$. Let
\begin{equation*}
T= \left\{ v_0,v_3,v_4,v_7,v_8,...,v_{4k-1},v_{4k},v_{4k+3},v_1',v_4',v_5',v_{8}',v_9',...,v_{4k}'. v_{4k+1}' \right\}    
\end{equation*}
It follows from Observation \ref{triangle} that no two elements of $T$ have mutual distance $2$. The result follows from Lemma \ref{maintool}.\\

Consider the case $t=3$. Let $w_{r-u}$ be so that $d(v_1,w_{r-u}) \geq r$ and $d(v_0,w_{r-u})=r-u$ for some $0 \leq u \leq 1$.

First, consider the case when $d(w_{r-u},v_{4k}') \geq 3$. Let 
\begin{equation*}
T= \left\{ v_0,v_1,v_4,v_5,...,v_{4k},v_{4k+1},v_3',v_4',v_7',v_8',...,v_{4k-1}',v_{4k}',w_{r-u}\right\}    
\end{equation*}
Assume for a contradiction that two vertices of $T$ have mutual distance $2$. It follows from Observation \ref{triangle} that one of them has to be $w_{r-u}$. Since for any $v,w$ in $G$, we have $d(v,w) \geq |d(v,v_0)-d(w,v_0)|$ and  $d(w_{r-u},v_{4k}') \geq 3$, it further follows that the other vertex would have to be $v_{4k}$ or $v_{4k+1}$. If we had $d(v_i,w_{r-u}) \leq 2$ for some $1 \leq i \leq 4k+1$, then $d(v_1,w_{r-u}) \leq d(v_1,v_i)+d(v_i,w_{r-u}) \leq 4k+2<r$, a contradiction. Hence, no two vertices of $T$ have mutual distance $2$, and $|T|=r$. The result follows from Lemma \ref{maintool}.

Next, consider the case $d(w_{r-u},v_{4k}') < 3$. Since $d(w_{r-u},v_{4k}') \geq |d(w_{r-u},v_0)-d(v_{4k}',v_0)|\geq 3-u \geq 2$, this means $u=1$ and $d(w_{r-1},v_{4k}') =2$. Hence there exists a vertex $a$ such that $a$ is neighbour of both $w_{r-1}$ and $v_{4k}'$. Moreover, clearly $d(a,v_0)=r-2$.

Consider two cases. If $d(a,v_{4k+3}) \geq 3$, let
\begin{equation*}
T= \left\{ v_0,v_3,v_4,v_7,v_8,...,v_{4k-1},v_{4k},v_{4k+3},v_1',v_4',v_5',v_8',v_9',...,v_{4k-4}',v_{4k-3}',v_{4k}',a \right\}.
\end{equation*}
Assume for a contradiction that two vertices of $T$ have mutual distance $2$. It follows from Observation \ref{triangle} that  one of them has to be $a$. Since for any $v,w$ in $G$, we have $d(v,w) \geq |d(v,v_0)-d(w,v_0)|$ and $d(a,v_{4k+3}) \geq 3$ and $d(a,v_{4k}')=1$, the other has to be $v_{4k-1}$ or $v_{4k}$. Since $d(a,v_{i}) \leq 2$ for some $3 \leq i \leq 4k$, we find $$d(v_3,v_{4k}') \leq d(v_3,v_i)+d(v_i,a)+d(a,v_{4k}) \leq (4k-3)+2+1<r,$$ a contradiction. Hence, no two vertices of $T$ have mutual distance $2$ while $|T|=r$. The result follows from Lemma \ref{maintool}.

Next, consider the case $d(a,v_{4k+3}) < 3$. By the triangle inequality, we have $d(a,v_{4k+3}) \geq 2$, so that $d(a,v_{4k+3}) =2$. Hence, there exists a vertex $b$, such that $b$ is neighbour of both $a$ and $v_{r}$. Consider 
\begin{equation*}
U= \left\{ v_0,v_1,v_2,v_3,...,v_{4k+3},v_{1}',v_2',...,v_{4k}',a,b \right\}    
\end{equation*}

We have $|U|=8k+6=2r$. Consider the auxiliary graph $H$ on $V(H)=U$ in which two vertices are connected if their distance in $G$ is precisely $2$. $H$ is union of two disjoint cycles of length $r$, first being $v_0,v_2,...,v_{4k+2},b,v_{4k}',...,v_2'$, and second being $v_1,v_3,...,v_{4k+3},a,v_{4k-1}',...,v_1'$. Indeed, the only non trivial relationships needed to prove that $H$ is union of two disjoint cycles of length $r$ are
\begin{equation*}
d(b,v_{4k+1}),d(b,v_{4k}),d(a,v_{4k+2}),d(a,v_{4k+1}),d(a,v_{4k}),d(a,v_{4k-1}) \geq 3 .
\end{equation*}
If any of these distances was at most $2$, we could find a path of length at most $r-1$ from $v_3$ to $v_{4k}'$.
The result follows from Lemma \ref{cycletool}.\\

Finally, consider the case $t=1$. 

\begin{claim}\label{foury}
Assume $r \geq 4$, $r=4k+3$ and $t=1$. Further assume there are four distinct vertices $y_1,y_2,y_3,y_4$ such that no two out of them have mutual distance $2$, and $d(v_0,y_i) \leq 3$ for $i=1,2,3,4$. Then we have $n \geq  2 \lceil \frac{r \delta}{2}  \rceil$.
\end{claim}

\begin{proof}[Proof of Claim \ref{foury}]
Let
\begin{equation*}
T= \left\{ y_1,y_2,y_3,y_4,v_6,v_7,v_{10},v_{11},...,v_{4k+2},v_{4k+3},v_6',v_7',v_{10}',v_{11}',...,v_{4k-2}',v_{4k-1}',v_{4k+2}' \right\}.
\end{equation*}
It follows by Observation \ref{triangle} that no two vertices of $T$ have mutual distance $2$. We also have $|T|=r$. The result follows by Lemma \ref{maintool}.
\end{proof}

We return to the proof of Proposition \ref{oneval} in the case $t=1$.

First consider the case when $v_2$ is not a center of $G$. Then there exists a vertex $c$ such that $d(v_2,c) \geq r+1$, and by the triangle inequality $d(v_0,c) \geq r-1$ and $d(v_3,c) \geq r$. We consider two cases: if $d(v_0,c)=r$, we could have chosen $c$ in place of $v_{r-t}'$ (as $d(v_3,c) \geq r$) and pass to a case $t=0$ which we already solved. If, on the other hand, $d(v_0,c)=r-1$, then let $v_0=v_0''',v_1''',...,v_{r-1}'''=c$ be a path of length $r-1$ from $v_0$ to $c$. No two out of $v_3,v_2,v_3''',v_2'''$ can have mutual distance $2$ by Observation \ref{triangle}, using that $d(v_2,c) \geq r+1$. Hence, we conclude by using Claim \ref{foury} for $y_1=v_3$, $y_2=v_2$, $y_3=v_3'''$, $y_4=v_2'''$.

Next, consider the case that $v_2$ is a center of $G$. Then we have $v_0,v_1,v_4,v_5$ such that $3 \geq d(v_2,v_0),d(v_2,v_1),d(v_2,v_4),d(v_2,v_5)$ and no two out of $v_0,v_1,v_4,v_5$ have mutual distance $2$. Now start the proof again with $v_0^\dagger:=v_2$ instead of $v_0$ (choosing some vertices $v_r^\dagger$ and $(v'_{r-t^{\dagger}})^{\dagger}$ in place of $v_r$, and $v_{r-t}'$). If $t^\dagger\neq1$, then the conclusion follows as before. If $t^\dagger=1$, then we can find four distinct vertices $y_1=v_0$, $y_2=v_1,$ $y_3=v_4$, $y_4=v_5$ such that no two out of them have mutual distance $2$, and $d(v_2,y_i) \leq 3$ for $i=1,2,3,4$. We conclude with Claim \ref{foury}.

This finishes the proof of Proposition \ref{oneval}.
\end{proof}

\section{General problem for girth $g \geq 5$}

We first establish Theorem \ref{upbound} using Lemma \ref{generaltool}.

\begin{proof}[Proof of Theorem \ref{upbound}]
We will find large enough collection of vertices $T$ such that no two non-adjacent vertices of $T$ are at mutual distance less than $2k-1$. The result then follows by Lemma \ref{generaltool}. 

Let $v_0$ be a center of $G$, $v_r$ a vertex with $d(v_0,v_r)=r$, and $v_0,v_1,...,v_{r-1},v_r$ a path of length $r$ in $G$ from $v_0$ to $v_r$. If we had $r \leq 2k$, we know the formula holds, so assume $r \geq 2k$. We let $v_{r-t}'$ be a vertex such that $d(v_{2k},v_{r-t}') \geq r$ and denote $d(v_0,v_{r-t}')=r-t$ for some $0 \leq t \leq 2k$. Further, let $v_0=v_0',v_1',...,v_{r-t}'$ be a path of length $r-t$ from $v_0$ to $v_{r-t}'$.

Let 
\begin{align*}
T= \left\{ v_{2ki} \, \, \, | \, \, \, 0 \leq i \leq \left\lfloor \frac{r}{2k} \right\rfloor \right\} \cup \left\{ v_{2ki+1} \, \, \, | \, \, \, 0 \leq i \leq \left\lfloor \frac{r}{2k} \right\rfloor -1 \right\} \\ \cup \left\{ v_{2ki}' \, \, \, | \, \, \, 1 \leq i \leq \left\lfloor \frac{r}{2k} \right\rfloor -1 \right\} \cup \left\{ v_{2ki+1}' \, \, \, | \, \, \, 1 \leq i \leq \left\lfloor \frac{r}{2k} \right\rfloor -2 \right\}.
\end{align*}

It follows from Observation \ref{distinct} that the above is a disjoint union. It follows from Observation \ref{triangle} that no two non-adjacent vertices of $T$ are at mutual distance less than $2k-1$. Hence, we conclude by Lemma \ref{generaltool} that $n \geq |T| \delta(\delta-1)^{k-2} \geq (\frac{2r}{k}-6)\delta(\delta-1)^{k-2}$.
\end{proof}

We prove the next proposition using an idea similar to one of Erd\H{o}s, Pollack, Pach and Tuza \cite{Erdos}. Its most important corollary is Theorem \ref{specialgirths}.

\begin{prop}\label{lobound}
Denote by $f(g,\delta)$ for $\delta \geq 2$, $g \geq 3$ the minimum number of vertices in the graph of girth at least $g$ and minimum degree $\delta$. Then for any $r > \frac{g}{2}$, there exists a connected graph $G$ on $n = \lceil \frac{2r}{g} \rceil f(g,\delta) $ vertices of girth at least $g$, minimum degree $\delta$ and radius at least $r$.
\end{prop}

\begin{proof}
Let $H$ be a connected graph such that $|V(H)|=f(g,\delta)$, the minimum degree of $H$ is $\delta$ and the girth of $H$ is at least $g$. As $\delta>1$, we know $H$ contains some cycle. Let $v,w$ be two neighbouring vertices of $H$ such that the edge $vw$ is part of some cycle. Let $H'$ be the (still connected) graph obtained by deleting the edge $vw$ from $H$. By the girth condition, we have $d_{H'}(v,w) \geq g-1$.

Take $\lceil \frac{2r}{g} \rceil $ identical disjoint copies of $H'$, called $H_1',...,H_{\lceil \frac{2r}{g} \rceil}'$, with vertices $v_1,...,v_{\lceil \frac{2r}{g} \rceil}$ and $w_1,...,w_{\lceil \frac{2r}{g} \rceil}$, and connect $v_i$ to $w_{i+1}$, where $w_{\lceil \frac{2r}{g} \rceil+1}=w_1$. The resulting graph has radius at least $r$, girth at least $g$, minimum degree $\delta$ and $\lceil \frac{2r}{g} \rceil f(g,\delta)$ vertices.
\end{proof}

Theorem \ref{specialgirths} follows easily.

\begin{proof}[Proof of Theorem \ref{specialgirths}] We know (see \cite{cages}) that when $\delta-1$ is a prime power, then $f(6,\delta) \leq 2(\delta^{2}-\delta+1)$, $f(8,\delta) \leq 2(\delta^{3}-2\delta^{2}+2\delta)$, and $f(12,\delta) \leq 2((\delta-1)^{3}+1)(\delta^{2}-\delta+1)$. Hence the result follows directly from Proposition \ref{lobound} by taking $n_i=\lceil \frac{i}{3} \rceil f(6,\delta)$, $n_j=\lceil \frac{j}{4} \rceil f(8,\delta)$, $n_k=\lceil \frac{k}{6} \rceil f(12,\delta)$.
\end{proof}

Finally, we prove Proposition \ref{erdos}.

\begin{proof}[Proof of Proposition \ref{erdos}]
Let $v_0$ be a center of our graph, $v_r$ a vertex with $d(v_0,v_r)=r$ and $v_0,...,v_r$ a path of length $r$.

Consider the sets $Q(v_0),...,Q(v_r)$, defined for each $v_i$ on our geodesic to be the points at distance at most $k$ from $v_i$. Every vertex in our graph is in at most $2k+1$ of these sets, so in particular some of these sets contains no more than $(2k+1)c \delta^{k-1}$ vertices. 

Also, it follows same as in the proof of Lemma \ref{generaltool} that each vertex has at least $\delta(\delta-1)^{k-2}$ vertices at distance at most $k-1$ from it. Hence, as all edges from these vertices are included in $Q(v_i)$, we get that for every $i$ the subgraph induced by $Q(v_i)$ has at least $\frac{1}{2}\delta^{2}(\delta-1)^{k-2}$ edges.

Putting this together, we get a connected graph of girth at least $2k$ on at most $(2k+1)c \delta^{k-1}$ vertices with at least $\frac{1}{2}\delta^{2}(\delta-1)^{k-2}$ edges.
\end{proof}

\section*{Acknowledgements}

The authors would like to thank their PhD supervisor professor B\'ela Bollob\'as for his support.

\appendix
\section{Useful observations}\label{observ}
We formulate some observations used throughout the proof in the following general setting. 

Let $G$ be a graph with $n$ vertices and radius $r$. We take $v_0$ to be some fixed center of $G$. We let $v_r$ be a vertex such that $d(v_0,v_r)=r$, and let $v_0,v_1,...,v_r$ be a path of length $r$ from $v_0$ to $v_r$.

Fix an integer $m\in\{1,\dots, r-1\}$, and let $v'$ be a vertex such that $d(v_m,v') \geq r$. Then let $t\geq 0$ be such that $d(v_0,v')=r-t$ , and let $v_0=v_0',v_1',...,v_{r-t}'=v'$ be a path of length $r-t$ form $v_0$ to $v'=v_{r-t}'$.

\begin{obs}\label{bound}
We have $t \leq m$.
\end{obs}

\begin{proof}[Proof of Observation \ref{bound}]
Assume for contradiction that we had $t>m$. Then by a triangle inequality, $d(v_m,v_{r-t}') \leq d(v_m,v_0)+d(v_0,v_{r-t}') =m+(r-t)<r$, a contradiction.
\end{proof}

\begin{obs}\label{triangle}
For any $m \leq i \leq r$ and any $0 \leq j \leq r-t$, we have $d(v_i,v_j') \geq d(v_m,v_{r-t}')+m+t+j-r-i$, and for any $i<m$ and any $0 \leq j \leq r-t$, we have $d(v_i,v_j') \geq d(v_m,v_{r-t}')+i+j+t-m-r$. Moreover, in either of these cases, we also have $d(v_i,v_j') \geq |i-j|$.
\end{obs}

\begin{proof}[Proof of Observation \ref{triangle}]
For the case $m \leq i \leq r$, note that $d(v_m,v_{r-t}') \leq d(v_m,v_i)+d(v_i,v_j')+d(v_j',v_{r-t}') =(i-m)+d(v_i,v_j')+(r-t-j)$. Rearranging gives the result.

For the case $i<m$, note that $d(v_m,v_{r-t}') \leq d(v_m,v_i)+d(v_i,v_j')+d(v_j',v_{r-t}') =(m-i)+d(v_i,v_j')+(r-t-j)$. Rearranging gives the result.

For the last claim, note that by triangle inequality $d(v_i,v_j') \geq |d(v_i,v_0)-d(v_0,v_j)'|=|i-j|$.
\end{proof}

\begin{obs}\label{distinct}
We can not have $v_i=v_i'$ for any $i > \frac{m+r-t-d(v_m,v_{r-t}')}{2}$, and we can not have $v_i=v_j'$ for any $i \neq j$.
\end{obs}

\begin{proof}[Proof of Observation \ref{distinct}]
Assume that $v_i=v_i'$ for some $r-t \geq i > \frac{m+r-t-d(v_m,v_{r-t}')}{2}$. Then we obtain contradiction, as $d(v_i,v_i')>0$ by Observation \ref{triangle}.

We can not have $v_i=v_j'$ for any $i \neq j$, since $d(v_0,v_i)=i \neq j =d(v_0,v_j')$. 
\end{proof}

\end{document}